\documentclass[12pt]{amsart}
\usepackage{geometry}                
\usepackage{float}
\usepackage{graphicx}
\usepackage{amsfonts,amsthm,amsmath,amssymb,latexsym, amscd, euscript}
\usepackage{graphicx,color}
\usepackage[all]{xy}
\usepackage{epsfig}
\usepackage{labelfig}

\usepackage{epstopdf}
\theoremstyle{plain}
\newtheorem{thm}{Theorem}[section]
\newtheorem{cor}[thm]{Corollary}
\newtheorem{prop}[thm]{Proposition}
\newtheorem{lem}[thm]{Lemma}

\newtheorem{rem}[thm]{Remark}
\newtheorem{exam}[thm]{Example}

\newtheorem{definition}[thm]{Definition}

\DeclareMathOperator{\Mob}{\mbox{\rm{M\"ob}}}
\DeclareMathOperator{\Int}{\rm Int}
\DeclareMathOperator{\Ext}{\rm Ext}
\DeclareMathOperator{\Chat}{\widehat{\mathbb{C}}}

\title{Handlebodies  of infinite genus and  Schottky groups}
\author{Ara Basmajian}
\address{The Graduate Center, CUNY \\ 365 Fifth Ave., N.Y., N.Y., 10016 and Hunter College, CUNY \\ 695 Park Ave., N.Y., N.Y., 10065, USA}
\email{abasmajian@gc.cuny.edu}
\thanks{A.B. Partially supported by a PSC-CUNY Grant and Simons Collaboration Grant  (359956, A.B.)}
\author{Katsuhiko Matsuzaki}
\address{Department of Mathematics, School of Education, Waseda University,
Shinjuku, Tokyo 169-8050, Japan}
\email{matsuzaki@waseda.jp}
\thanks{K.M. Partially supported by Japan Society for the Promotion of Science (KAKENHI 23K25775 and 23K17656)}

\begin{document}
\begin{abstract}
In this paper, we begin an investigation of infinite genus handlebodies, infinitely generated Schottky groups, and related uniformization questions by giving appropriate definitions for them. There are uncountably many topological types of infinite genus surfaces with non-planar ends. We show that any such surface and 
any infinite genus handlebody can be topologically uniformized by an infinitely generated classical Schottky group. 
We next show that an infinite genus Riemann surface with non-planar ends admitting a bounded pants decomposition can be 
quasiconformally uniformized by a classical Schottky group. In addition, the conformal  equivalence class of 
such a uniformization is unique.
If the assumption of bounded pants decomposition is removed we supply examples of such Riemann surfaces that do not admit 
a quasiconformal uniformization by a Schottky group. 
\end{abstract}

\maketitle


\section{Introduction and statement of results} 
\label{sec: intro}
  In this paper, we investigate infinitely generated Schottky groups, infinite genus handlebodies, and questions such as whether a Riemann surface of infinite topological type has a  Schottky uniformization, and whether 
it has a quasiconformal uniformization. We note that the collection of infinite genus topological surfaces with non-planar ends form a rich (uncountable) class of examples. We show that any topological surface of infinite genus and no planar ends can be realized as the quotient of a Schottky group. 
Of course,  we first must define what it means to be an infinitely generated Schottky group and an infinite genus handlebody.  Along the way, one natural class of Kleinian groups that arises are the so-called 
Schottky-like groups.

Consider a disjoint  finite  set of Jordan curves  $\{C_i ,C_i^{\prime}\}_{i}$ with common exterior  in the complex plane each   paired by a M\"obius transformation $g_i$ taking the exterior of $C_i$ to the interior of 
$C_i^{\prime}$. The group generated by the $\{g_i\}$ is called a (finitely generated) Schottky group -- it is arguably one of the simplest  Kleinian groups.  Among its salient features   are that  it is free, purely loxodromic, has totally disconnected   limit set,  the   Jordan curve exteriors  form a fundamental domain,  and the quotient  $3$-manifold is 
a finite genus handlebody. 
If the finite set of curves is replaced with an infinite set many of these features are no longer automatic. As a result,  we make some basic assumptions on the curve systems. In order to parse what definitions are necessary, we introduce an intermediate notion of Kleinian groups we call {\it Schottky-like}. When the group is finitely generated,  Schottky and Schottky-like are the same notions. In order to make precise what structures we are considering we first define some terms. 

A countable collection of Jordan curves 
$\{C_i ,C_i^{\prime}\}_{i \in \mathbb N}$ in \(\widehat{\mathbb{C}}\) is said to be an {\em admissible configuration} if
\begin{itemize} 
\item the Jordan curves in $\{C_i ,C_i^{\prime}\}_{i \in \mathbb N}$ are pairwise disjoint, and there is a 
point \(z_0 \in \widehat{\mathbb{C}}\) for which no Jordan curve in this set separates any other Jordan curve 
from \(z_0\), and
\item no point on one of the Jordan curves in $\{C_i ,C_i^{\prime}\}_{i \in \mathbb N}$ is the limit of points from  
the other Jordan curves.
\end{itemize}

We let \(\Ext(C)\) denote the component of 
$\widehat{\mathbb C}\setminus C$ that contains $z_0$, and \(\Int(C)\) the 
other component. When \(z_0= \infty \) these are the usual notions of interior and exterior of a Jordan curve.
We denote the set of all Jordan curves in the admissible configuration by 
$\mathcal{C}$, and the common exterior of all the 
curves by $\Ext (\mathcal{C})$. We remark that $\Ext (\mathcal{C})$ is connected but not necessarily open.
The largest open set contained in $\Ext (\mathcal{C})$ (i.e. the interior of $\Ext (\mathcal{C})$) is denoted by 
$\Ext (\mathcal{C})^{\circ}$.

\begin{definition}\label{Schottky}
{\rm  
A Kleinian group $G < \Mob$ is said to be {\it Schottky-like}  
if there exist M\"obius transformations $g_i \in G$
$(i \in \mathbb N)$,
and an admissible configuration of Jordan curves 
$\mathcal{C}=\{C_i ,C_i ^{\prime}\}_{i \in \mathbb N}$ with
$G =\langle g_i :i \in \mathbb N\rangle$, and 
$g_i(\Ext(C_i))=\Int( C_i^{\prime})$. If in addition
(i) the open set $\Ext (\mathcal{C})^{\circ}$  is a fundamental domain for $G$ and (ii) the limit set $\Lambda (G)$ is totally disconnected then  $G$ is called a {\it Schottky}  group.
A Schottky-like or Schottky group is {\it classical} if the Jordan curves are circles. }
\end{definition}  


First, we consider topological uniformization of infinite genus handlebodies. Similar to infinitely generated Schottky groups,
we need to define such handlebodies appropriately. We assume that a Schottky group $G$ acts also on
the $3$-dimensional upper half-space $\mathbb H^3$ properly discontinuously and consider the quotient of
$\mathbb H^3 \cup \Omega(G)$ by $G$, which is a $3$-manifold with boundary. Here, 
$\Omega(G)=\widehat{\mathbb C} \setminus \Lambda(G)$ is the region of discontinuity of $G$.

\vskip10pt
\noindent{\bf Theorem A} (Propositions \ref{prop: Schottky group iff handlebody} and \ref{prop: end space same}){\bf .}
{\it  An infinite genus handlebody can be topologically uniformized by an infinitely generated classical Schottky group. 
Conversely,  if  $G$ is an infinitely generated Schottky group 
satisfying condition $(\ast)$, then  $(\mathbb{H}^{3} \cup \Omega(G))/G$ is an infinite genus handlebody. Moreover, for a Schottky group  $G$
 the end space of the Riemann surface $\Omega(G)/G$ and that of the handlebody $(\mathbb H^3 \cup \Omega(G))/G$ are homeomorphic. }
\vskip10pt 
In Theorem A condition $(\ast)$  states that 
for any  subsequence $\{i_k\} \subset \mathbb N$,  if 
$C_{i_k}  \rightarrow x \ \mbox{\rm then }  C_{i_k}^{\prime}  \rightarrow x, \ {\rm as}\ k \to \infty.$

As a consequence of Theorem A, we have 
that  a handlebody is determined by its genus and space of ends (Corollary \ref{cor: handlebody determined}).
For a locally compact Hausdorff space $X$ in general, the space of ends of $X$ is defined as $\overline X \setminus X$ for
the Freudenthal compactification $\overline X$ of $X$ (which is also known as the Stoilow--Ker\'ekj\'art\'o compactification
for a Riemann surface $X$).

In our next theorem, we address the question of Schottky uniformization of Riemann surfaces. 

\vskip10pt
\noindent{\bf Theorem B} (Theorems \ref{thm: classical schottky uniformize} and \ref{thm: qc uniformization}){\bf .}
{\it  Let $R$ be a Riemann surface of infinite genus with no planar ends.  Then 
  there is a classical Schottky group $G$ such that 
 $\Omega(G)/G$ is homeomorphic to $R$.
  If in addition, $R$ has a bounded pants decomposition then there is a quasiconformal homeomorphism from $R$ to 
$\Omega(G)/G$.
}
\vskip10pt

Moreover, the uniformization of such a surface $R$ with a bounded pants decomposition
is rigid under conformal homeomorphisms.

\vskip10pt
\noindent{\bf Theorem C} (Theorem \ref{uniqueness}){\bf .}
{\it  Suppose $G$  is a classical Schottky group with respect to the admissible configuration of circles $\mathcal{C}$ where   
 $\Omega(G)/G$ has  a bounded pants decomposition. Let 
 $G^{\prime}$ be a Schottky-like group with admissible configuration of circles $\mathcal{C}^{\prime}$. If 
 $h:\Omega(G)/G \to \Omega(G')/G'$ is a conformal mapping which sends the  free homotopy classes of simple closed curves determined by 
$\mathcal C$  to those of $\mathcal{C}^{\prime}$,  then $G$ and $G'$ are conjugate by a M\"obius transformation.}

\vskip10pt
As a corollary to Theorem C we have
\vskip10pt 
\noindent{\bf Corollary D} (Corollary \ref{cor: uniqueness of handlebody}){\bf .}
{\it   Let $M$ be an infinite genus handlebody with  complete
meridian system $\mathcal A$ uniformized topologically by
a classical Schottky group $G$ with  admissible configuration $\mathcal C$ corresponding to $\mathcal A$. If $\Omega(G)/G$ admits a bounded pants decomposition,
then any conformal structure on $\partial M$ quasiconformally equivalent to $\Omega(G)/G$
relative to $\mathcal C$ determines the unique hyperbolic structure on $M$.}
\vskip10pt

Any finite genus Riemann surface is quasiconformally equivalent to a Riemann surface uniformized by a classical Schottky group (in fact, the Riemann surface itself is conformally equivalent to a (not necessarily classical) Schottky group. However, we show in section \ref{sec: no Schottky uniformization: examples} that there exist infinite genus  Riemann surfaces that do not admit a quasiconformal uniformization by a Schottky group. We in fact prove a necessary condition that such a Riemann surface 
admit a quasiconformal uniformization.

\vskip10pt 
\noindent{\bf Theorem E} (Examples \ref{Ex1} and \ref{Ex2}){\bf.}
{\it In each topological equivalence class of infinite genus  surface with non-planar ends, 
there exists a Riemann surface $S$ with associated Fuchsian group of the second kind such that  $S$ 
does not admit a quasiconformal uniformization by a Schottky group.
Moreover, there exists an infinite genus Riemann surface $R$ with only one end which is non-planar
that does not admit a quasiconformal uniformization by a Schottky group.}

\vskip10pt
\noindent{\bf Plan of the paper:}  In section \ref{sec: intro} we define the notions of an admissible system, Schottky and Schottky-like groups, and then state our main theorems. In section \ref{sec: definitions and basics} we lay out the needed definitions and basics.  In particular, we derive elementary but important properties of  a classical Schottky group.  In section \ref{sec: nest sequences} we supply examples to show the necessity of the definition of a Schottky group and the independence of the conditions that the exterior of the circles form a fundamental domain and that the limit set be totally disconnected.  We also define the space of ends and show that the space of ends can be identified with a particular closed subset of the limit set of the Schottky group. 

In section \ref{sec: handlebodies} we define what an  infinite genus handlebody is and discuss some of its properties. 
Then, we show that any infinite genus handlebody is topologically uniformized by a Schottky group.
In section \ref{sec: uniformization I} we prove a  topological Schottky uniformization theorem for all infinite genus surfaces with non-planar ends.  This is followed by  section \ref{sec: uniformization II} where we show that such an infinite genus Riemann surface having a bounded pants decomposition can be quasiconformally uniformized by a classical Schottky group. 
We also prove that such a uniformization is unique in its conformal equivalence  class.
Finally, in section \ref{sec: no Schottky uniformization: examples} we take up the question of whether there exist infinite genus Riemann surfaces that do not have a quasiconformal uniformization by a Schottky group. 

We note that sections \ref{sec: uniformization I}, \ref{sec: uniformization II}, and 
\ref{sec: no Schottky uniformization: examples} address analogues uniformization questions for infinite genus handlebodies.
It is not difficult to translate the conditions on a Riemann surface to
those on handlebodies of hyperbolic structure.
\vskip5pt

\centerline{\bf Acknowledgements}  This work began  during a one month visit to Tokyo, Japan in 2013. A.B. would like to thank  the Waseda Institute for Advanced Study  for  support during that visit as well as Katsuhiko Matsuzaki and Hiroshige Shiga for their generous hospitality.

\section{Definitions and basics} \label{sec: definitions and basics}
Denote the orientation preserving M\"obius transformations of 
$\widehat{\mathbb{C}}$ by $\Mob$. 
For a discrete ({\it Kleinian}) group $G < \Mob$, we let 
$\Lambda(G) \subset \Chat$ denote the {\it limit set} and 
$\Omega(G)=\widehat{\mathbb C}\setminus \Lambda(G)$ the {\it region of discontinuity}.
An open subset $F \subset \Omega(G)$ is by definition 
a {\it fundamental domain} for $G$ if every point of 
$\Omega(G)$ is equivalent to some point of the closure
$\overline F$ and no two points of $F$ are equivalent under $G$.

In section \ref{sec: intro} we defined what an admissible system is as well as the notion of a  Schottky or  Schottky-like group associated to the  admissible  system. 
In this section some basics on admissible curve systems are given
in connection with Schottky-like group. We  refer to Maskit \cite[Sect.VIII.A]{M}
for some of these properties.

\begin{prop}\label{prop: basics0}
Let $\mathcal{C}=\{C_i ,C_i ^{\prime}\}_{i \in \mathbb N}$ be an 
admissible configuration of Jordan curves and let $G =\langle g_i :i \in \mathbb N\rangle$
be a subgroup of $\Mob$ whose generators $g_i$ satisfy $g_i(\Ext(C_i))=\Int( C_i^{\prime})$. 
Then the following hold
\begin{enumerate}

\item $G$ is purely loxodromic and free on a countable number of generators.

\item $\Ext (\mathcal{C})$ is precisely invariant under the identity in $G$.

\item If $\Ext (\mathcal{C})^{\circ}$ is non-empty, then $G$ is a Schottky-like group
with respect to $\mathcal C$.

\item If all Jordan curves in $\mathcal C$ are circles, then $G$ is a classical
Schottky-like group
with respect to $\mathcal C$.
\end{enumerate}
\end{prop}

\begin{proof}
(1) Let $G_n=\langle g_i :1 \leq i  \leq n \rangle$. Then, $G_n$ is a finitely generated Schottky group and
 the family $\{G_n\}$ gives the exhaustion of $G$, i.e. $G=\bigcup_n G_n$.
 Since the finitely generated Schottky group $G_n$ is purely loxodromic and free, so is $G$.
 
(2) Let $\mathcal{C}_n=\{C_i ,C_i ^{\prime}\}_{1 \leq i \leq n}$. Then $\Ext (\mathcal{C}_n)$ contains 
 $\Ext (\mathcal{C})$, and $\Ext (\mathcal{C}_n)$ is precisely invariant under the identity in $G_n$.
 This implies that $\Ext (\mathcal{C})$ is precisely invariant under the identity in $G$.

(3) Suppose that $G$ is not discrete. Then there exists a sequence $g_n \in G \setminus \{\rm id\}$ that converges to the identity, and 
$g_n(z)$ converges to $z$ locally uniformly on $\Ext (\mathcal{C})^{\circ}$. However, this contradicts the fact that $\Ext (\mathcal{C})$ is precisely
invariant under the identity. Thus, $G$ is discrete, and is Schottky-like.

(4) We regard $G$ acting on the upper half-space $\mathbb H^3$. 
When $C_i$ and $C_i ^{\prime}$
are circles, they span totally geodesic planes in the hyperbolic space $\mathbb H^3$, and
define a non-empty convex domain bounded by them. 
This domain is precisely invariant under the identity
for the same reason as before, which implies that $G$ is discrete, and is classical Schottky-like.
\end{proof}

\begin{lem} \label{lem: basics2}  Suppose 
  $G =\langle g_i :i \in \mathbb N\rangle$ is a classical Schottky-like group with respect to 
  the admissible configuration of circles $\mathcal{C}=\{C_i ,C_i ^{\prime}\}_{i \in \mathbb N}$. Then
  the accumulation set $\Lambda'$ of the configuration circles is a closed subset of the limit set $\Lambda(G)$.
  \end{lem}
  
  \begin{proof}
  There is a limit point of $G$ in $\Int(C)$ for each $C \in \mathcal C$. Since the spherical diameters of
  the configuration circles in $\mathcal C$ tend to zero, any accumulation point of the configuration circles
  is an accumulation point of the limit points. From the fact that the limit set of $G$ is closed,
  the statement follows. 
  \end{proof}

\begin{prop} \label{prop: basics1}  Suppose 
  $G =\langle g_i :i \in \mathbb N\rangle$ is a classical Schottky group with respect to 
  the admissible configuration of  circles $\mathcal{C}=\{C_i ,C_i ^{\prime}\}_{i \in \mathbb N}$. Then the fundamental domain $\Ext (\mathcal{C})^{\circ}$ for $G$ is connected.
\end{prop}
 
\begin{proof}
 Let $\Lambda'$ be the accumulation set of the configuration circles for the classical Schottky group $G$.
 We observe that $\Ext (\mathcal{C})^{\circ}=\Ext (\mathcal{C}) \setminus \Lambda'$, 
 where $\Lambda'$ is
 totally disconnected by Lemma \ref{lem: basics2} and the property of the limit set $\Lambda(G)$
 of a Schottky group $G$. Suppose that $\Ext (\mathcal{C})^{\circ}$ is not connected.
 Then, we can find an open set $A \subset \widehat{\mathbb C}$ such that $\Ext (\mathcal{C})^{\circ}$ is contained 
in the union of
$A$ and its exterior $\Ext A=\Int(\widehat{\mathbb C} \setminus A)$, and 
both the intersections of these open sets with
$\Ext (\mathcal{C})^{\circ}$ are not empty. 
Because the closed discs 
$\overline{\Int(C_i)}$ and $\overline{\Int(C'_i)}$ are not contained in
the boundary $\partial A$, by filling them into $\Ext (\mathcal{C})^{\circ}=\Ext (\mathcal{C}) \setminus \Lambda'$,
 we would have that $\widehat{\mathbb C} \setminus \Lambda'$ is contained 
 in the union of $A$ and $\Ext A$.
 However, $\widehat{\mathbb C} \setminus \Lambda'$ is connected since $\Lambda'$ is totally disconnected.
 Hence, we obtain that $\Ext (\mathcal{C})^{\circ}$ is connected by this contradiction.
 \end{proof}

 
 \section{Nested sequences of  circles and completeness}
\label{sec: nest sequences}
 Suppose $G=\langle g_i: i \in \mathbb N \rangle$ is a classical Schottky-like group with respect to the circles in $\mathcal{C}=\{C_i ,C_i ^{\prime}\}_{i \in \mathbb N}$, and set $D=\Ext (\mathcal{C})^{\circ}$. In this section we give a combinatorial description  for the nested (infinite) sequences of $G$-translates of circles in  
 $\mathcal{C}$. Denote the set of symmetric generators, $\{g_i^{\pm 1}\}$, by  $\mathcal{S}$, and 
 $\Gamma_{\mathcal{S}} (G)$ to be  the  associated  Cayley graph.  Hence the vertices correspond to elements of $G$, and there is an edge from  $h_1$ to  $h_2$ if there exists 
 $g \in \mathcal{S}$ so that  $h_1 g=h_2$ (note that we use right multiplication by $g$). Thus 
 $\Gamma_{\mathcal{S}} (G)$ is an infinite valence Cantor tree, and we denote the word length of $h \in G$ by $|h|$. 
  The vertices and edges  of the graph have a direct correspondence to the action of the classical Schottky-like group which in turn  makes encoding nested sequences of circles transparent. 
  
  \begin{eqnarray*}
\text{{\bf VERTICES:}} \qquad \qquad \qquad \qquad  G& \longleftrightarrow & 
\text{G-translates of D} \\
h& \mapsto& hD \\
\text{{\bf EDGES:}} 
\qquad \{(h_1, h_1g):  g \in \mathcal{S}\}& 
\longleftrightarrow & 
\{\text{G-translates of circles in $\mathcal{C}$} \}\\
(h_1, h_1 g)& \mapsto& h_1D \cap h_1gD=h_1(D \cap gD) \\
\end{eqnarray*}

To summarize: Vertices correspond to translates of $D$, and edges  correspond to circles at the intersection of two translates of $D$. 

  Consider the set of G-translates of circles in 
$\mathcal{C}$. An infinite sequence of such circles $\{A_i\}_{i \in \mathbb N}$ is said to be {\it nested} if $A_i$ separates  $A_{i-1}$ from $A_{i+1}$, for each $i$. It is 
{\it maximal} if for each $i$ between $A_i$ and $A_{i+1}$ there are no translates of circles.                                                                               
Clearly every nested sequence of circles is contained in a maximal nested sequence.  We'll say that two nested sequences of circles are {\it equivalent} if they only differ by a finite number of circles. The property of being maximal is preserved by this equivalence relation. 
Our combinatorial model leads to

\begin{prop}
The set of asymptotic classes of geodesic rays  in 
$\Gamma_{\mathcal{S}}(G)$ are in one to one correspondence with the set of equivalence classes of maximal nested sequences of circles. 
\end{prop}

\begin{proof}
A geodesic ray $\gamma$ on  $\Gamma_{\mathcal{S}} (G)$ corresponds to  a sequence of $G$-translates of $D$ where consecutive translates differ by one of the generators. Since a geodesic ray cannot double back this gives a maximal nested sequence of circles. Since  
$\Gamma_{\mathcal{S}} (G)$ is a tree, a geodesic ray asymptotic to $\gamma$ eventually follows the same path 
as $\gamma$, and hence corresponds to an equivalent maximal nested sequence of circles. 

Conversely, a maximal nested sequence of circles corresponds to 
a unique asymptotic class of geodesic rays.
\end{proof}

We consider conditions for $D=\Ext (\mathcal{C})^{\circ}$ to be a fundamental domain.
We can also find such conditions in Maskit \cite[Sect.VIII.A]{M}.

\begin{lem} \label{lem: fund domain}
Suppose $G =\langle g_i :i \in \mathbb N\rangle$ is classical Schottky-like  
with respect to admissible configuration of circles $\mathcal{C}=\{C_i ,C_i ^{\prime}\}_{i \in \mathbb N}$. Then
  
 \begin{enumerate}
  \item $\Ext (\mathcal{C})^{\circ}$ is a fundamental domain if and only if every nested sequence of translated circles, $\{h_i C_{k_i}\}_{i \in \mathbb N}$, where $h_i \in G$ and $C_{k_i} \in \mathcal{C}$, has 
  spherical diameter that goes to zero,  as $i  \rightarrow \infty$.
  \item  If the hyperbolic distance in
  $\mathbb{H}^3$ between any two planes that span circles in 
  $\mathcal{C}$ is uniformly bounded from below by a positive constant, then 
  $\Ext (\mathcal{C})^{\circ}$ is a fundamental domain.
 
\end{enumerate}
\end{lem}

\begin{proof}
(1) Suppose that there exists a nested sequence of translated circles whose
spherical diameters do not go to zero. Then, these circles accumulates to a circle $C$.
In this case, we can see that $\Int(C)$ is included in $\Omega(G)$ and
no point in $\Int(C)$ is equivalent to a point 
in $\overline{\Ext (\mathcal{C})}$ under $G$. This shows that $\Ext (\mathcal{C})^{\circ}$ is 
not a fundamental domain for $G$. 

Conversely, suppose that for every nested sequence of translated circles 
$\{h_i C_{k_i}\}_{i \in \mathbb N}$, their 
spherical diameters go to zero as $i \rightarrow \infty$. We may assume that the nested sequence is maximal.
The accumulation point of this sequence is a singleton, which belongs to the limit set $\Lambda(G)$.
We observe that $\Lambda(G)$ comprises the set of accumulations of all such nested sequences of
translated circles and all translates of $\Lambda'$. Therefore, for every point $\zeta$ in 
$\Omega(G) \setminus \overline{\Ext (\mathcal{C})}$,
there exist a maximal nested sequence $\{h_i C_{k_i}\}_{i \in \mathbb N}$ and an $i \in \mathbb N$ 
uniquely such that
$\Int(h_iC_{k_i})$ contains $\zeta$ but $\Int(h_{i+1}C_{k_{i+1}})$ does not.
This shows that every point in $\Omega(G)$ is equivalent to $\overline{\Ext (\mathcal{C})}$ under $G$.
Since $\Ext (\mathcal{C})$ is precisely invariant under the identity in $G$ by Proposition \ref{prop: basics0},
$\Ext (\mathcal{C})^{\circ}$ is a fundamental domain for $G$.

(2) The argument is reduced to (1). 
For every maximal nested sequence of translated circles $\{h_i C_{k_i}\}_{i \in \mathbb N}$,
let $H_i$ be the hyperbolic plane in $\mathbb H^3$ that spans $h_i C_{k_i}$. Then, the hyperbolic distance
between $H_i$ and $H_{i+1}$ is not less than the infimum of
the hyperbolic distances 
between any distinct two planes that span circles in $\mathcal{C}$. By the assumption, 
this infimum is bounded from below by some $\delta>0$. This implies that 
the hyperbolic distance
between $H_i$ and $H_{i+j}$ for $j \in \mathbb N$ is not less than $j\delta$, and hence the 
spherical diameter of $h_i C_{k_i}$ goes to zero as $i \rightarrow \infty$.
\end{proof}

\begin{cor}\label{fundamental}
Let $G$ be a classical Schottky group with admissible configuration of circles $\mathcal C$.
Then for any $M >0$, the number of circles in 
$\{g(C) : C \in \mathcal{C},\ g \in G\}$ with spherical diameter greater than
$M$ is finite. 
\end{cor}

\begin{proof} 
Suppose there are infinitely  many distinct  
$C_i \in \mathcal{C}$ and $h_i \in G$ where $h_i (C_i)$ has spherical diameter bounded from below. Since the sphere has finite area and since the interiors of the $C_i$ must be disjoint,  it must be that for the index $i$ large the $h_i (C_i)$ are nested and thus converge to a circle $C$. This implies that 
$\Ext (\mathcal{C})^{\circ}$ is not a fundamental domain for $G$ by (1) of Lemma \ref{lem: fund domain}. 
However this contradicts to the requirement for a Schottky group.
\end{proof}

Consider a compact exhaustion of a topological manifold $M$.
An {\it end}  of $M$ is an equivalence class of nested connected complementary components of the exhaustion,
where  two such, $U_{1} \supset U_{2} 
\supset U_{3}\supset \cdots \supset U_{n} \supset \cdots$
and $V_{1} \supset V_{2} 
\supset V_{3}\supset \cdots \supset V_{n} \supset \cdots$ are equivalent
if for every $i$ there exists $j$ such that $U_i \supset V_j$ and $V_i \supset U_j$.
The set of ends $\mathcal{E}$ has a natural topology -- a basis element associated to 
the finite sequence $U_{1} \supset U_{2} 
\supset U_{3}\supset \cdots \supset U_{n} $ being the set of ends that  have a representative of the equivalence class that begins with this finite sequence. In the case of orientable surfaces, the space of ends $\mathcal{E}$ is homeomorphic to a  closed subset of the
Cantor set where  the infinite genus  (non-planar) ends
 $\mathcal{E_{\infty}} $ form a closed subspace. 
An end is called {\it planar} if it has some planar neighborhood $U$. 
 A theorem due to Ian Richards \cite{R} says that the surface is determined 
 topologically by its genus, and the double space $\left(\mathcal{E}, \mathcal{E_{\infty}}\right)$. 

It is not difficult to see that the ends of a Schottky surface 
$\Omega(G)/G$ are 
 non-planar (infinite genus) (see also Lemma \ref{lem: geometric criteria}). 
We have seen in Lemma \ref{lem: basics2} that the accumulation set of the configuration curves, denoted $\Lambda^{\prime}$,  
is a closed subset of the totally disconnected limit set $\Lambda(G)$ when $G$ is a classical Schottky group. We prove the identification of the end space of $\Omega(G)/G$ with $\Lambda'$.

\begin{lem}\label{end}
Let $G$ be a Schottky group with admissible curve configuration $\mathcal{C}=\{C_i ,C_i^{\prime}\}_{i \in \mathbb N}$. Assume that 
for any  subsequence $\{i_k\} \subset \mathbb N$,
\begin{equation}\label{*}
{\rm if }\ C_{i_k}  \rightarrow x \ \mbox{\rm then } \ C_{i_k}^{\prime}  \rightarrow x, \ {\rm as}\ k \to \infty.
\tag{$\ast$}
\end{equation}
Then $\Lambda^{\prime}$ is homeomorphic to the end space  of 
$\Omega(G)/G$. In particular, the hypothesis is satisfied if $G$ is classical and the hyperbolic distances 
between the totally geodesic planes that span $C_i$ and $C_i^{\prime}$ are uniformly bounded  
from above for all $i$.
\end{lem}

\begin{proof} Let $\Lambda^{\prime}$ be the set of accumulation points of the configuration curves.  
We take a compact exhaustion $\{K_n\}$ of $\widehat{\mathbb {C}}\setminus\Lambda^{\prime}$ such that
the complement of $K_n$ consists of finitely many topological discs $D_{n,j}$ that has  non-empty 
intersection with $\Lambda^{\prime}$ and their boundaries $\beta_{n,j}=\partial D_{n,j}$
are disjoint from the configuration curves. Under the projection $\Omega(G) \to \Omega(G)/G$,
we take the images of $K_n$, $D_{n,j}$, and $\beta_{n,j}$ restricted to $\Omega(G) \cap \overline{\Ext(\mathcal C)}$, which are denoted by
$S_n$, $E_{n,j}$, and $\gamma_{n,j}$, respectively. Then $S_n$ are compact bordered surfaces in $\Omega(G)/G$
with the boundary components $\gamma_{n,j}$.
By virtue of the assumption that $C_{i_k} \rightarrow x$ if and only if 
$C_{i_k}^{\prime} \rightarrow x$ $(k \to \infty)$, we can see that
$E_{n,j}$ are end neighborhoods of $\Omega(G)/G$ for all sufficiently large $n$. Namely,
the boundary of $E_{n,j}$ is $\gamma_{n,j}$.

Any point on the end space of $\widehat{\mathbb {C}}\setminus\Lambda^{\prime}$ determined by a sequence $D_{1,j_1} \supset D_{2,j_2} \supset \cdots$
corresponds to a point on the end space of $\Omega(G)/G$ determined 
by a sequence $E_{1,j_1} \supset E_{2,j_2} \supset \cdots$. This gives a bijection between these end spaces.
Moreover, by the correspondence of the bases of neighborhoods $\{D_{n,j}\}$ and $\{E_{n,j}\}$,
this bijection is in fact a homeomorphism between the end spaces of
 $\widehat{\mathbb {C}}\setminus\Lambda^{\prime}$ and $\Omega(G)/G$. Finally since  
 $\Lambda^{\prime}$ is a closed  totally disconnected set, the end space of its complement,
 $\widehat{\mathbb {C}}\setminus\Lambda^{\prime}$, is precisely $\Lambda^{\prime}$.
\end{proof}

Our next examples show that the  hypotheses that 
$\Ext (\mathcal{C})^{\circ}$  is  a fundamental domain  and the  limit set is  totally disconnected 
  in the definition of a Schottky group are independent.

\begin{exam}[{\bf $D$ a fundamental domain but fat limit set}]
{\rm In this example we show that a Schottky-like group $G$ can have $\Ext (\mathcal{C})^{\circ}$  as   a fundamental domain but its limit set may not be totally disconnected. Since the choice of Jordan curves for the group $G$ does not effect the limit set, we in fact show that $G$ can not be Schottky with respect to any curve system. 

Choose an admissible  set of  circles 
$\mathcal{C}=\{C_i ,C_i ^{\prime}\}_{i \in \mathbb N}$ contained in the unit disc that accumulate on the  whole boundary  with the property that the hyperbolic distance between any two totally geodesic planes 
in $\mathbb H^3$ that span distinct circles is uniformly bounded from below. For each $ {i \in \mathbb N}$, choose a  M\"obius transformation $g_{i}$ that maps $C_i$ to $C_i ^{\prime}$ and takes $g_i(\Ext(C_i))=\Int( C_i^{\prime})$; setting 
$G=\langle g_i : i \in \mathbb N\rangle$ note that $G$ is Schottky-like. Clearly 
$\Lambda (G)$ contains the unit circle and hence condition (ii) in 
Definition \ref{Schottky} 
is not satisfied.  On the other hand, by
Lemma \ref{lem: fund domain}, $\Ext (\mathcal{C})^{\circ}$ is a fundamental domain. }
\end{exam}

\begin{exam}[{\bf Totally disconnected limit set but $D$ not a fundamental domain}] 
\label{incomplete}

{\rm We will construct  a Fuchsian Schottky group $G$  of the second kind which when viewed as acting on the Riemann sphere has an admissible set of circles where each circle  intersects the upper half-plane in a geodesic. Using hyperbolic geometry we give simple geometric criteria to show that this admissible set of circles has translates that form a nested sequence of circles  converging to a circle. Hence the limit set is totally disconnected but $\Ext (\mathcal{C})^{\circ}$ will not be a fundamental domain by Lemma  \ref{lem: fund domain}. 

\begin{figure}[t]
\begin{center}
\AffixLabels{\centerline{\epsfig{file =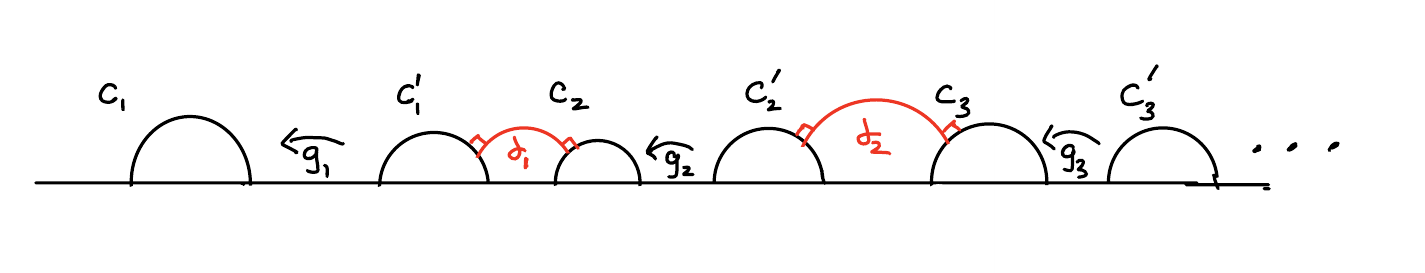,width=15.0cm} }}
\vspace{-0.5cm}
\end{center}
\caption{Admissible circles.} \label{fig: FIGCircles}
\end{figure}

We choose a set of admissible circles 
$\mathcal{C}=\{C_i ,C_i ^{\prime}\}_{i \in \mathbb N}$ as 
depicted in Figure \ref{fig: FIGCircles}. These circles are perpendicular to 
$\mathbb{R}$, converge to $\infty$ to the right, and when viewed as geodesics in the upper half-plane are a positive distance apart. Let $\delta_{i}$ be the orthogeodesic 
from 
$C_{i} ^{\prime}$ to  $C_{i+1}$,
for each 
$i \in \mathbb N$, and assume that 
$\sum \ell(\delta_{i}) < \infty$, where $\ell$ stands for the hyperbolic length. We next construct the group. 
For each $i\in \mathbb{N}$, choose $g_{i} \in \text{PSL}(2, \mathbb{R})= \Mob(\mathbb R)$, so that
$g_i(\Ext(C_i^{\prime}))=\Int( C_i)$.  By postcomposing $g_{i}$ 
by a hyperbolic element whose axis is $C_i$, we may 
further assume that 
$g_{i}(\delta_{i})$ smoothly meets   
$\delta_{i-1}$ in $C_i$.
We continue to call the element 
$g_i$. Hence, 
$\delta_{i-1}$ smoothly connects with $g_{i}(\delta_{i})$
to extend the geodesic segment.
We continue this process to  obtain the nested  sequence of 
geodesics, 
$$
\{C_{1}, g_{1}(C_{2}), 
g_{1}g_{2}(C_{3}),\ldots, 
g_{1}g_{2}\cdots g_{i}(C_{i+1}),\ldots\}
$$
which by construction has a geodesic ray that is perpendicular to all the circles in the nested sequence. See Figure \ref{fig:NestedCircles}. 

We set 
$G=\langle g_i : i \in \mathbb N\rangle$. 
Now, since 
$\ell(\delta_i)=\ell(g_{1}g_{2}\cdots g_{i}(\delta_{i}))$ for each $i$,
we may conclude
$$\sum \ell(g_{1}g_{2}\cdots g_{i}(\delta_{i}))< \infty,$$
and thus the nested  sequence we created converges to a geodesic, call it $C_{\infty}$.  Viewing the action of $G$ on the Riemann sphere,  we see that $G$ is a  classical 
Schottky-like group.  Now, since $G$ is Fuchsian, and acts discontinuously on a non-empty open subset of 
$\widehat{\mathbb{R}}$, we have that the limit set of $G$ is a totally disconnected set in
$\widehat{\mathbb{R}}$. On the other hand,  by item (1) of Lemma \ref{lem: fund domain}  
$\Ext(\mathcal C)^{\circ}$ is not a fundamental domain. 
}
\end{exam}

\begin{figure}[t]
\begin{center}
\AffixLabels{\centerline{\epsfig{file = 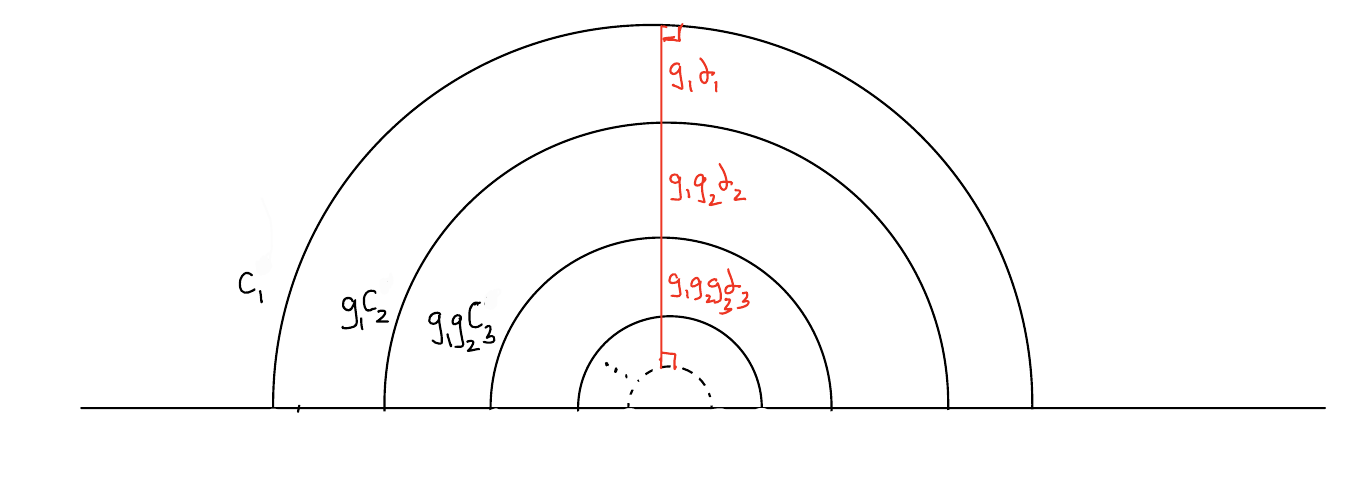,width=15.0cm}}}
\vspace{-1.0cm}
\end{center}
\caption{Nested circles.} \label{fig:NestedCircles}
\end{figure}

\begin{rem}
{\rm
Let $G$ be an infinitely generated purely hyperbolic Fuchsian group of the second kind having a wandering 
half-plane (i.e. precisely invariant under the identity) bounding an interval of discontinuity. The group $G$ in Example \ref{incomplete} satisfies this condition.
The associated hyperbolic surface is the union of its convex core and the
half-plane. Regarding the existence and the construction of such a hyperbolic surface, see \cite{Ba} and \cite{B-S}.
By a theorem of 
Purzitsky \cite{Pur}, $G$ is a Schottky group with {\em some} admissible 
configuration of curves $\mathcal{C}$. See Button \cite{Bu} to expect that $G$ is further classical possibly with a different
$\mathcal{C}$, which are perpendicular to the real line.  
However, in Example \ref{incomplete}, we construct a particular $\mathcal{C}$
for which $G$ is not Schottky.
}
\end{rem}



\section{Infinite genus handlebodies}\label{sec: handlebodies}

Let ${\bf B}^{3} \subset \mathbb{R}^{3}$ be the closed $3$-ball, and let $M$ be a  3-manifold with boundary. 
$M$ is said to be a {\it handlebody of genus $g$}
($0 \leq g \leq \infty$) if $M$ is constructed  by the disc sum $M={\bf B}^3 \cup \bigcup T_i$
of $g$ solid tori $\{T_i\}_{i=1}^g$ along disjoint discs 
$D_{i} \subset \partial {\bf B}^3 \setminus E$ where $E \subset \partial {\bf B}^3$ is the accumulation set of the discs 
$\{D_i\}_{i=1}^g$ and is totally disconnected 
(see Figure \ref{fig: Handlebody}).
In the case of finite genus, see Gross \cite{Gr} for the
decomposition by the disc sum.
The genus $g$ is finite if and only if the accumulation set $E$ is empty.

A compressing disc in $M$ is a properly embedded disc $A$ with $\partial A \subset \partial M$. 
If $M$ is a handlebody of genus $g$
($0 \leq g \leq \infty$) in the above definition, then there exist $g$ compressing discs 
$\mathcal{A}=\{A_i\}_{i=1}^{g}$ (complete meridian disc system) in the solid tori $\{T_i\}_{i=1}^g$ 
so that if $N_i$ is an open cylindrical  
neighborhood of $A_i \subset T_i$ with $N_i \cap N_j =\emptyset$, for $i \neq j$, then 
$M \setminus \bigcup N_i$ is homeomorphic to ${\bf B}^3 \setminus \Lambda^{\prime}$, where 
$\Lambda^{\prime} \subset \partial {\bf B}^{3}$ is the (possibly empty) accumulation set of the compressing discs
$\{A_i\}_{i=1}^{g}$. This is also totally disconnected (the handlebody is finite genus if and only if this accumulation set is empty), and in fact,
$\Lambda^{\prime}$ is homeomorphic to $E$. 

A subtle point here for infinite genus handlebodies 
is that $\Lambda^{\prime}$ is the accumulation set of the compressing discs before they are identified to yield a $1$-handle whereas $E$ is the accumulation set of the discs where the solid tori are attached. In the group case,
$\Lambda^{\prime}$
corresponds to the accumulation set of the admissible configuration curves for the Schottky group $G$ whereas $E$ corresponds to the end space of the quotient by $G$.

\begin{figure}[t]
\begin{center}
\AffixLabels{\centerline{\epsfig{file =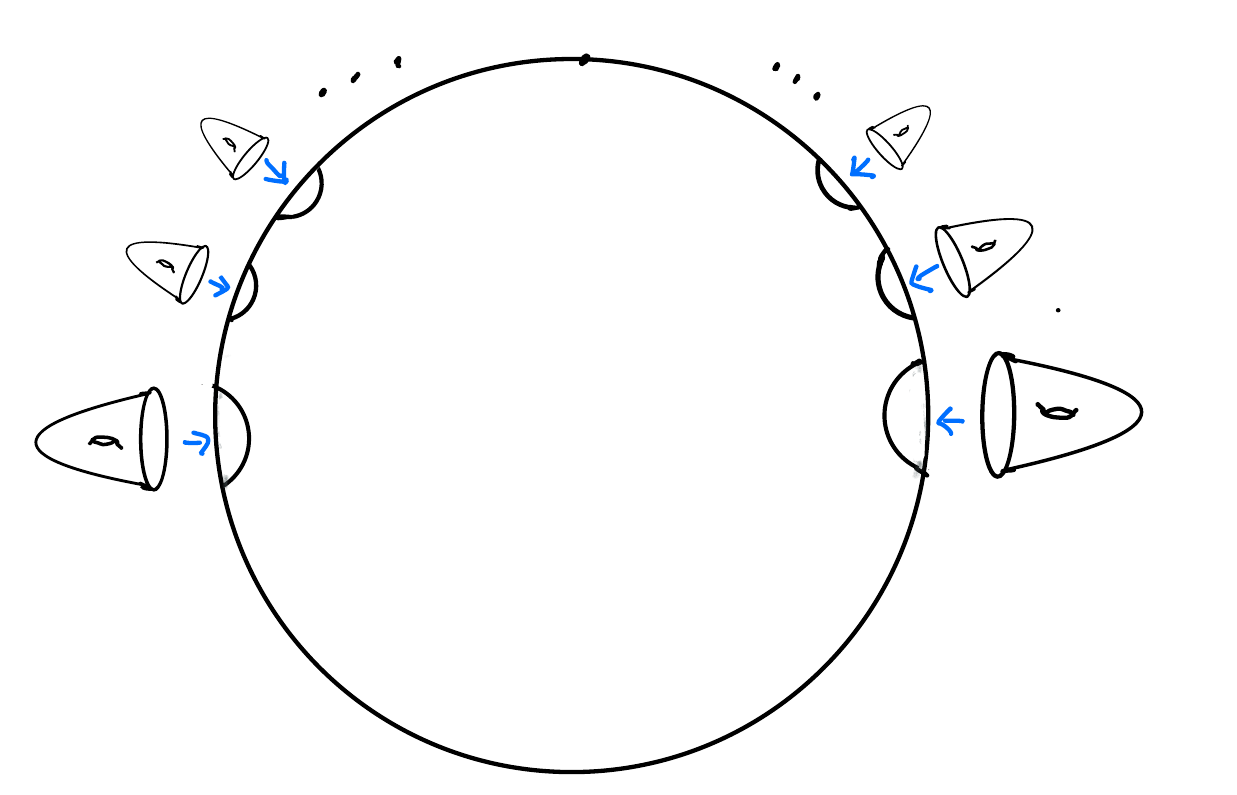,width=13.0cm} }}
\vspace{.0cm}
\end{center}
\caption{Handlebody Decomposition}\label{fig: Handlebody}

\end{figure}

\begin{prop} \label{prop: Schottky group iff handlebody}
If $G$ is a Schottky group with admissible curve configuration ${\mathcal C}=\{C_i,C_i'\}_{i=1}^g$ $(0 \leq g \leq \infty)$
satisfying condition \eqref{*} in Lemma \ref{end}, then 
$\left(\mathbb{H}^{3} \cup \Omega (G) \right)/G$ is a handlebody.
Conversely, a  handlebody can be topologically uniformized by such a classical Schottky group. 
\end{prop}

\begin{proof} 
If $G$ is a Schottky group then cutting the $3$-manifold quotient 
$\left(\mathbb{H}^{3} \cup \Omega (G) \right)/G$ along discs $\{A_i\}_{i=1}^{g}$
whose boundaries are given by the admissible curves $\{C_i,C_i'\}_{i=1}^g$
at infinity which are identified by the generators yields the three ball ${\bf B}^3$ with $1$-handles attached. 
Lemma \ref{end} implies that the accumulation set (if it is non-empty)
$\Lambda'$ of these handles is homeomorphic to the set $E$ that need to be deleted from
$\partial {\bf B}^{3}$ to take the disc sum with the solid tori $\{T_i\}_{i=1}^g$. 
Hence, $\left(\mathbb{H}^{3} \cup \Omega (G) \right)/G$ is a handlebody. 

On the other hand, suppose $M$ is an infinite genus handlebody (the finite genus  case is well-known)
and let $E$ be its space of ends. 
We will build a classical Schottky group $G$ for which $\left(\mathbb{H}^{3} \cup \Omega (G) \right)/G$ is homeomorphic to $M$.
The proof is very similar to the proof of Theorem \ref{thm: classical schottky uniformize} in the next section 
so we just supply an outline.  
By cutting at the complete meridian disc system $\mathcal{A}=\{A_i\}_{i=1}^{g}$ associated with the disc sum 
with the solid tori $\{T_i\}_{i=1}^{g}$ constructing $M$, we have that
$M \setminus \bigcup N_i$ is homeomorphic to ${\bf B}^3 \setminus \Lambda^{\prime}$. Here,
$\{N_i\}$ defines disjoint closed Jordan domains $\{D_i,D_i'\}_{i \in \mathbb N}$ on $\partial {\bf B}^3$ with
$\Lambda^{\prime}$ as its accumulation set, which is homeomorphic to $E$. By applying a self-homeomorphism of
${\bf B}^3$, we may assume that all $D_i$ and $D_i'$ are round discs.
Let $C_i=\partial D_i$ and $C_i'=\partial D_i'$, and take the classical Schottky
group $G$ with the admissible configuration of circles $\mathcal C=\{C_i,C_i'\}_{i=1}^g$.
Then,
$\left(\mathbb{H}^{3} \cup \Omega (G) \right)/G$ is homeomorphic to $M$.
\end{proof}

\begin{prop} \label{prop: end space same}
Let $M$ be a handlebody topologically uniformized by a classical Schottky group $G$. 
Then, the end space of the Riemann surface $\Omega(G)/G$ and that of $M = (\mathbb H^3 \cup \Omega(G))/G$ are homeomorphic. 
\end{prop}

\begin{proof}
Consider a compact exhaustion $\{S_n\}$ of $\Omega(G)/G$.
We may assume that the boundaries of the compact bordered surface $S_n$ are dividing closed curves $\gamma_{n,j}$.
Then, each $\gamma_{n,j}$ is the frontier of a connected component $E_{n,j}$ of the complement $(\Omega(G)/G) \setminus S_n$.
Moreover, each $\gamma_{n,j}$ spans a dividing compression disc $D_{n,j}$ in $M$. Let $H_n$ be
the portion of $M$ divided by $\{D_{n,j}\}$ that contains $S_n$, and let
$\widehat E_{n,j}$ be a connected component of the complement $M \setminus H_n$ 
that contains $E_{n,j}$.
We see that each $H_n$ is a compact handlebody and $\{H_n\}$ gives a compact exhaustion of 
$M$.

Any point on the end space of $\Omega(G)/G$ determined by a sequence $E_{1,j_1} \supset E_{2,j_2} \supset \cdots$
corresponds to a point on the end space of $M$ determined 
by a sequence $\widehat E_{1,j_1} \supset \widehat E_{2,j_2} \supset \cdots$. This gives a bijection between these end spaces.
Moreover, by the correspondence of the bases of neighborhoods $\{E_{n,j}\}$ and $\{\widehat E_{n,j}\}$,
this bijection is in fact a homeomorphism between the topological spaces.
\end{proof}

Since a handlebody can be realized topologically by the quotient of a Schottky group 
preserving the correspondence between the accumulation sets (Proposition \ref{prop: Schottky group iff handlebody}), Proposition \ref{prop: end space same} immediately implies the following corollary. 

\begin{cor} \label{cor: handlebody determined}
A handlebody is determined by its genus and space of ends. 
\end{cor}


\section{Classical Schottky uniformization I: Topological uniformization}\label{sec: uniformization I}
Topologically, every surface of infinite genus with no planar ends admits
classical Schottky uniformization. We see later that this condition on non-planar end is necessary.

\begin{thm} \label{thm: classical schottky uniformize}
Let $R$ be a  surface of infinite genus with no planar ends. Then there is a classical
Schottky group $G$ such that $\Omega(G)/G$ is homeomorphic to $R$.
\end{thm}

\begin{proof}
Since $R$ has no planar ends, the space of ends of $R$ is homeomorphic to
a closed subset $\Lambda'$ of the Cantor set. Put $\Lambda'$ on the real line $\mathbb R$ and
consider each bounded open interval $I$ of the complement $\mathbb R \setminus \Lambda'$.
Let $h$ be a hyperbolic M\"obius transformation in $\Mob(\mathbb R)={\rm PSL}(2,\mathbb R)$ 
having the two end points of $I$
as the fixed points and choose two circles $C$ and $C'$ in the Ford fundamental domain of $\langle h \rangle$ 
so that
(i) the closures of ${\rm Int}\,C$ and ${\rm Int}\,C'$ are disjoint; (ii) $C$ and $C'$ are orthogonal to
$\mathbb R$; (iii) the distance between $C$ and $C'$ and the distances from $C$ and $C'$ to the boundary of the
fundamental domain are relatively large to the radii of $C$ and $C'$.
Then take a hyperbolic element $g$ of $\Mob(\mathbb R)$ such that
$g({\rm Ext}(C))={\rm Int}(C')$ and define $g_n=h^n g h^{-n}$, $C_n=h^n(C)$ and $C'_n=h^n(C')$ 
for each $n \in \mathbb Z$. The accumulation points of these circles $\{C_n,C'_n\}_{n \in \mathbb Z}$
are exactly the two end points of $I$. See Figure \ref{fig: EndSpace}.

\begin{figure}[h]
\begin{center}
\AffixLabels{\centerline{\epsfig{file = 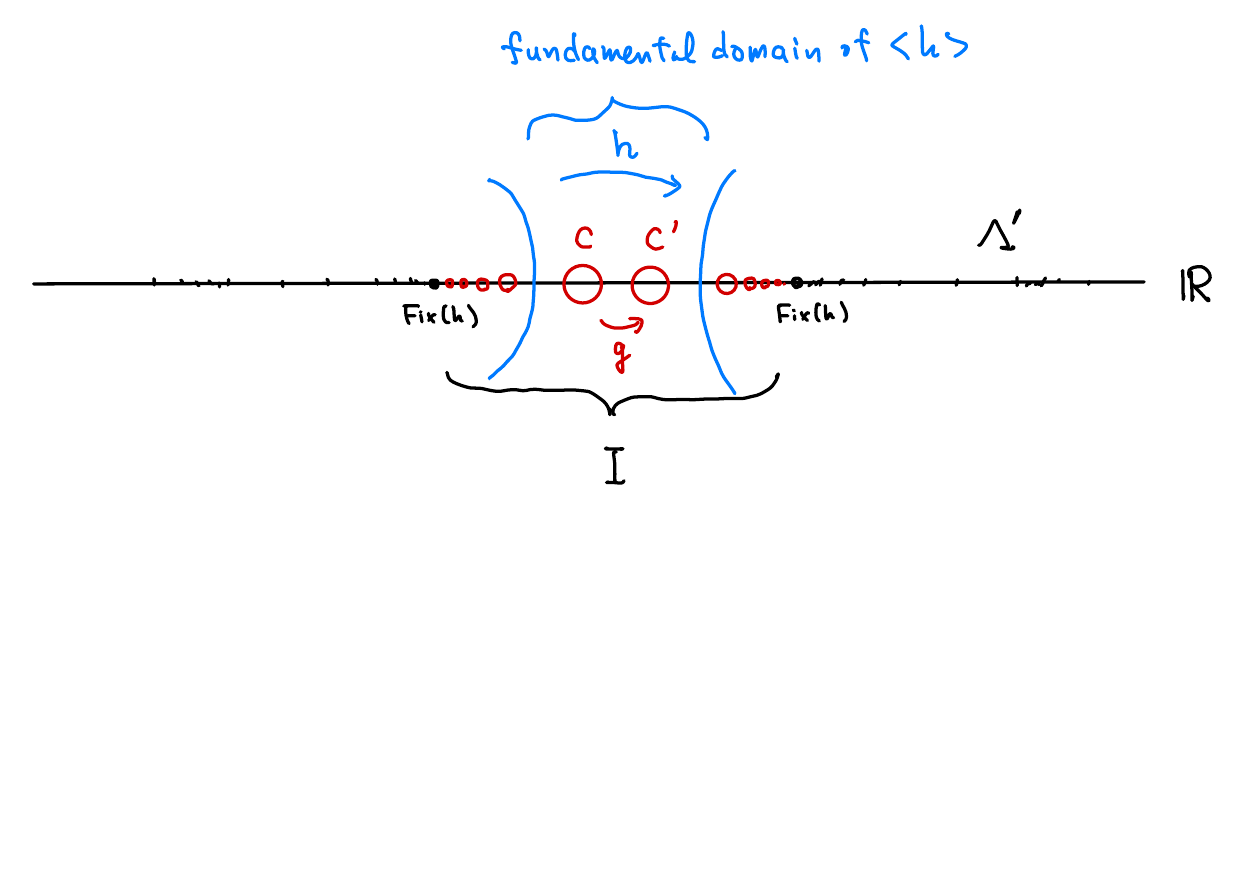,width=15.0cm} }}
\vspace{-5cm}
\end{center}
\caption{Admissible circle configuration realizing a prescribed end space.} \label{fig: EndSpace}

\end{figure}

We repeat the same construction for each bounded open interval $I=I_j$ of $\mathbb R \setminus \Lambda'$ 
with $j$ in a countable index set $J$
and obtain
$\{g_{j,n}\}_{n \in \mathbb Z} \subset \Mob(\mathbb R)$ and 
the circles $\{C_{j,n},C'_{j,n}\}_{n \in \mathbb Z}$ for each $j$.
Then we see that the set of accumulation points of the family of circles 
$\mathcal C=\{C_{j,n},C'_{j,n}\}_{j \in J,\ n \in \mathbb Z}$ coincides with $\Lambda'$.

Let $G$ be a subgroup of $\Mob(\mathbb R)$ generated by all 
$\{g_{j,n}\}_{j \in J, n \in \mathbb Z}$. Then, $G$ is discrete by Proposition
\ref{prop: basics0}. Moreover,
the interior
$$
\left(\bigcap_{j \in J,\ n \in \mathbb Z} {\rm Ext}\,(C_{j,n}) \cap {\rm Ext}\,(C'_{j,n})\right)^{\circ}
$$
is a fundamental domain for $G$. Indeed, 
we can choose the circles $C$ and $C'$ as in requirement (iii) above to meet 
the condition in Lemma \ref{lem: fund domain} (2).
Moreover, the limit set $\Lambda(G) \subset \mathbb R$ is totally
disconnected. Hence $G \subset \Mob(\mathbb R)$ is a classical Schottky (Fuchsian) group.
Since assumption \eqref{*} of Lemma \ref{end} is satisfied by 
the above construction of $\mathcal C$, the space of ends of the Riemann surface $\Omega(G)/G$ is homeomorphic to $\Lambda'$.
This completes the proof.
\end{proof}


\section{Classical Schottky uniformization II: Quasiconformal uniformization} \label{sec: uniformization II}
We show that any Riemann surface of a certain moderate geometric condition
is quasiconformally equivalent to classical Schottky uniformization.

\begin{definition}
{\rm
We say that a Riemann surface $R$ has a {\em bounded pants decomposition} if $R$ is composed by
a union of pairs of pants whose boundary geodesics have uniformly bounded lengths.
}
\end{definition}

\begin{thm}\label{thm: qc uniformization}
Let $R$ be a Riemann surface of infinite genus that admits
a bounded pants decomposition. 
Then there exists a classical Schottky group $G$
such that $\Omega(G)/G$ is quasiconformally equivalent to $R$.
\end{thm}

\begin{proof} 
Let $R=\bigcup P_n$ be a bounded pants decomposition.
We cut open $R$ by a family of simple closed geodesics $\{\gamma_i\}_{i \in \mathbb N} \subset \bigcup \partial P_n$
to obtain a planar domain $R'$ conformally equivalent to $R \setminus\{\gamma_i\}_{i \in \mathbb N}$.
The closure $\overline{R'}$ can be regarded as a bordered hyperbolic surface with geodesic boundary.
Since $R$ admits the bounded pants decomposition, if we choose an exhaustion of $R'$ by a sequence of
compact bordered hyperbolic surfaces $\{R'_n\}$ with geodesic boundary, then we can assume that
the lengths of boundary components of $R'_n$ are also uniformly bounded. Hence,
we see that $R'$ satisfies the assumption in Maitani and Taniguchi \cite[Th.1]{MaiTani}.
Then, as a special case of
the generalization of circular domain theorem of Koebe due to He and Schramm \cite{HS1,HS2}, 
there is a planar domain $D \subset \mathbb C$
conformally equivalent to $R'$ such that $\partial D$ consists of disjoint circles $\{C_i, C'_i\}_{i \in \mathbb N}$,
where $C_i$ and $C'_i$ correspond to $\gamma_i$.

For each $i \in \mathbb N$, the original identification of $C_i$ with $C'_i$ in $R$
defines a diffeomorphism $\iota_i:C_i \to C'_i$. On the other hand, 
we choose some
$g_i \in \Mob$ such that $g_i(\Int (C_i))=\Ext (C'_i)$ and that its axis  
is perpendicular to both the geodesic planes spanned by $C_i$ and $C'_i$
in the upper half-plane
${\mathbb H}^3$ with the hyperbolic metric. In general, $\iota_i:C_i \to C'_i$ and $g_i:C_i \to C'_i$ are different,
but we can claim the following.

\smallskip
\noindent
{\bf Claim.} For each $i \in \mathbb N$, 
the difference $g_i^{-1} \circ \iota_i:C_i \to C_i$ is a diffeomorphism whose derivative is
uniformly bounded and bounded away from zero independent of $i$.
\smallskip

For all $i \in \mathbb N$, we identify $C_i$ with $C'_i$ by $g_i$. This gives a Riemann surface $R_*$
possibly different from $R$ but quasiconformally equivalent to $R$.
We assume the above claim for the moment and verify this fact first. The proof of the claim is
given separately below. 

For each $\gamma_i$, the canonical collar in $R$, which can be taken by the collar lemma,
is conformally equivalent to an annulus $\{z \mid 1/r_i<|z|<r_i\}$. Note that $r_i>1$ is uniformly bounded away from 1
because the hyperbolic length of $\gamma_i$ is uniformly bounded. 
We denote the two sides of $\gamma_i$ in the canonical collar
by  $A_i \cong \{z \mid 1<|z|<r_i\}$ and $A_i' \cong \{z \mid 1/r_i<|z|<1\}$. 

We consider $g_i^{-1} \circ \iota_i:C_i \to C_i$ under the identification of $C_i$ with $\gamma_i=\{z \mid |z|=1\}$
and denote it by $h_i:\gamma_i \to \gamma_i$. 
We may assume that $h_i(1)=1$ by composing a rotation to $g_i$.
Define a map $H_i:A_i \to A_i$ by the linear interpolation between $h_i$ on $|z|=1$ and the identity on $|z|=r_i$.
More explicitly,
$$
H_i(te^{i\theta})=t\exp \left(i\int_0^\theta \left \{\frac{r_i-t}{r_i-1}(|h'_i(e^{i\rho})|-1)+1 \right \}\,d\rho \right)
\quad (1 \leq t \leq r_i,\ 0 \leq \theta < 2\pi).
$$
Since $r_i>1$ is uniformly bounded away from 1 and since $|h'_i(e^{i\rho})|$ is uniformly bounded and bounded away from zero,
we see that the maximal dilatation of $H_i$ is uniformly bounded.
Then a map of $R$ defined by $H_i$ on $A_i$ for all $i \in \mathbb N$ and by the identity elsewhere
is a quasiconformal homeomorphism onto $R_*$.

Let $G$ be the  subgroup of $\Mob$ generated by all such $g_i$.
By \cite[Th.2]{MaiTani}, the limit set
$\Lambda(G)$ is in the class $N_{D}$ (i.e. $\Omega(G) \in O_{AD}$, which will be discussed again later)
and in particular totally disconnected. Moreover, 
$D=\Ext(\mathcal C)^\circ$ is a fundamental domain for $G$.
Thus we see that $G$ is a classical Schottky group with admissible configuration  
${\mathcal C}=\{C_i, C'_i\}_{i \in \mathbb N}$.
\end{proof}

In the remainder of this section, we verify the claim used in the above proof.
Let $f:R' \to D$ be the conformal homeomorphism onto the circular domain $D$.
The restrictions of $f$ to the collar half-neighborhoods $A_i$ and $A'_i$ are specifically denoted by conformal embeddings
$f_i:A_i \to D$ and $f'_i:A'_i \to D$. By the reflection principle,
these maps extend to conformal homeomorphisms
$$
\tilde f_i:\gamma_i \cup A_i \cup A'_i \to N(C_i), \quad \tilde f'_i:\gamma_i \cup A_i \cup A'_i \to N(C'_i),
$$
where $N(C_i)$ is the union of $C_i$, $f_i(A_i)$ and its reflection with respect to $C_i$, and
$N(C'_i)$ is the union of $C'_i$, $f'_i(A'_i)$ and its reflection with respect to $C'_i$.
See Figure \ref{fig: EmdCollars}.
Then, $\tilde f'_i \circ \tilde f_i^{-1}:N(C_i) \to N(C'_i)$ is conformal, and
it coincides with the identification $\iota_i$ on $C_i$. We set $F_i=\tilde f'_i \circ \tilde f_i^{-1}$ and
consider the derivative $F'_i$ on $C_i$.

\begin{figure}[h]
\begin{center}
\AffixLabels{\centerline{\epsfig{file = 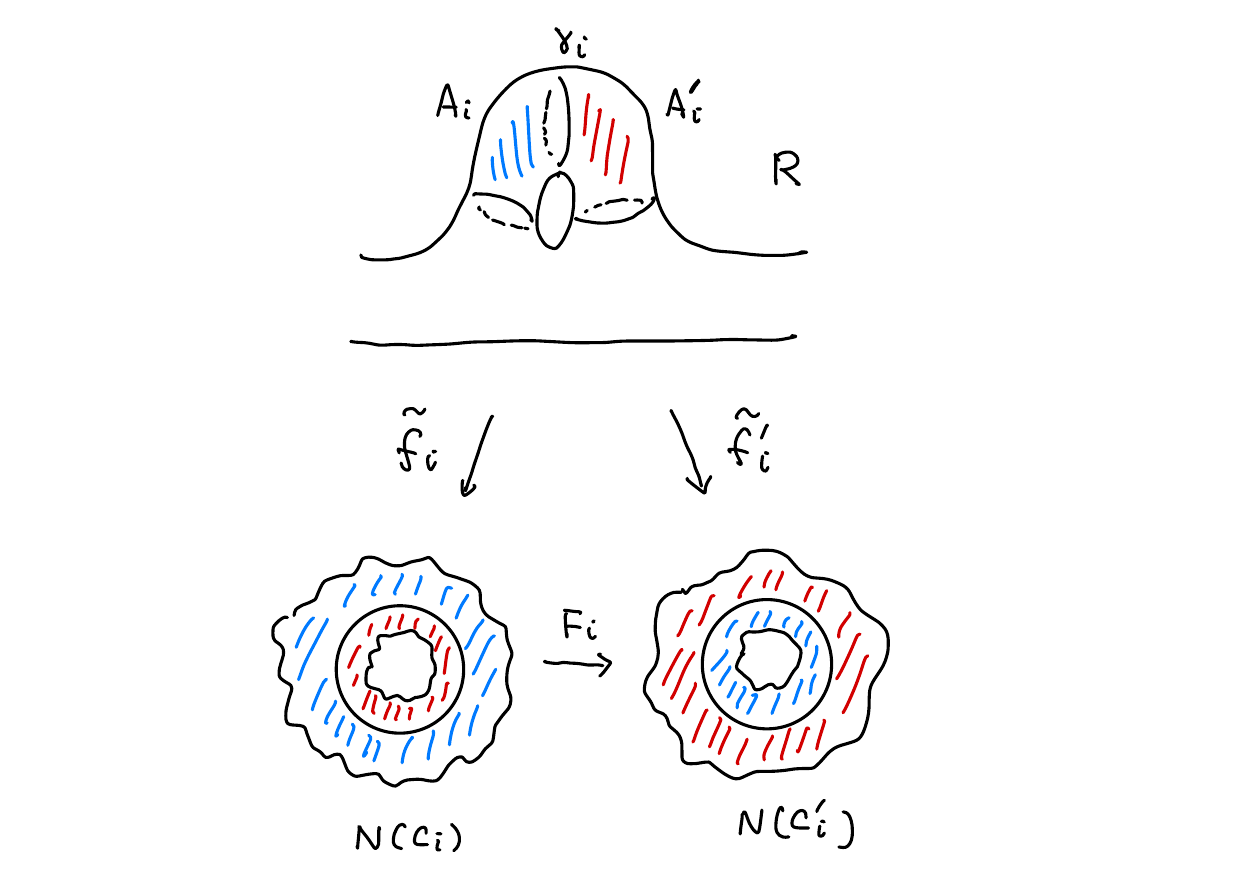,width=15.0cm} }}
\vspace{-1cm}
\end{center}
\caption{Embedding of collars into a circle domain.} \label{fig: EmdCollars}

\end{figure}

We begin with showing that the images $f_i(A_i)$ and $f'_i(A'_i)$ in $D$ contain round annuli with boundary components
$C_i$ and $C'_i$ respectively whose conformal moduli are uniformly bounded away from 0. Hereafter, by an annulus, we mean
a round annulus $A=\{z \mid r_1<|z-a|<r_2\}$ and denote its conformal modulus $\log(r_2/r_1)$ by $\rm mod(A)$.
The concentric circle $\{z \mid |z-a|=r\}$ in $A$ with $r/r_1=r_2/r$ is called the core circle of $A$.

\begin{lem}\label{grotzsch}
Let $\Delta$ be the unit disc and $W \subset \Delta$ a doubly connected domain bounded by
$\partial \Delta$ and a simple closed Jordan curve separating $0$ from $\partial \Delta$. 
Then $W$ contains an annulus $A$ 
one of whose boundary component is $\partial \Delta$ such that
$$
{\rm mod}(A) \geq -\log \mu^{-1}({\rm mod}(W)),
$$
where $\mu:(0,1) \to (0,\infty)$ is a decreasing function giving the conformal modulus of
the Gr\"otzsch extremal domain $\Delta \setminus [0,r]$ by $\mu(r)$.
\end{lem}

\begin{proof}
If $W$ does not contain a point $z \in \Delta$ with $|z|=r$, then ${\rm mod}(W) \leq \mu(r)$
by the Gr\"otzsch module theorem (see \cite[p.54]{LV}). This implies that every $z \in \Delta$ 
with $|z| >\mu^{-1}({\rm mod}(W))$ is contained in $W$. By taking the annulus $A$ consisting of all such $z$,
we obtain the statement.
\end{proof}

We also use the following general principle concerning the derivative and the moduli of a conformal map between annuli. 

\begin{prop}\label{annulus}
Let $A_1 \subset \mathbb C$ be an annulus whose core circle $C_1$ has radius $r_1>0$ and
$A_2 \subset \mathbb C$ an annulus whose core circle $C_2$ has radius $r_2>0$.
Suppose that there are positive constants $m, M>0$ such that 
$m \leq {\rm mod}(A_1)$ and ${\rm mod}(A_2) \leq M$. Then
any conformal map $f$ of $A_1$ into $A_2$  
with $f(C_1)=C_2$
satisfies 
$$
|f'(z)| \leq \frac{4M}{m} \frac{r_2}{r_1}
$$
for every $z \in C_1$.
\end{prop}

\begin{proof}
Consider the universal covers $p_1:\widetilde A_1 \to A_1$ and $p_2:\widetilde A_2 \to A_2$
defined by $\zeta \mapsto r_1 e^{i\zeta}$ and $\zeta \mapsto r_2 e^{i\zeta}$ respectively,
where $\widetilde A_1=\{\zeta \mid |{\rm Im} \zeta| <{\rm mod}(A_1)/2\}$ and
$\widetilde A_2=\{\zeta \mid |{\rm Im} \zeta| <{\rm mod}(A_2)/2\}$. 
Then $f$ lifts to a conformal homeomorphism
$\widetilde f:\widetilde A_1 \to \widetilde A_2$ such that $\widetilde f(\mathbb R)=\mathbb R$. 

For a point $z \in C_1$, choose its lift $\zeta \in \mathbb R$. Since $|f'(z)|=(r_2/r_1)|\widetilde f'(\zeta)|$,
we estimate $|\widetilde f'(\zeta)|$. By the Koebe one-quarter theorem (see \cite[p.56]{LV}) applied to the disc of center $\zeta$ and
radius ${\rm mod}(A_1)/2$, we see that
$$
\frac{1}{4}\, {\rm mod}(A_1) |\widetilde f'(\zeta)| \leq {\rm mod}(A_2).
$$
This implies that $|f'(z)| \leq (4M/m)(r_2/r_1)$.
\end{proof}

\noindent
{\it Proof of Claim.}
For each $i \in \mathbb N$, we consider the conformal map $F_i:N(C_i) \to N(C'_i)$. 
Let $r_i, r'_i>0$ be the radii of $C_i, C'_i$ respectively. 
We can choose an annular neighborhood $A(C'_i) \subset N(C'_i)$ with the core circle $C'_i$ such that
${\rm mod}(A(C'_i))=M$
for some constant $M>0$ independent of $i$. This is due to Lemma \ref{grotzsch}.
Moreover, we can choose an annular neighborhood $A(C_i) \subset F_i^{-1}(A(C'_i))$ with the core circle $C_i$ such that
${\rm mod}(A(C_i)) \geq m$ for some constant $m>0$ independent of $i$, which is also by Lemma \ref{grotzsch}.
Then, we apply Proposition \ref{annulus} to the conformal map $F_i$ on $A(C_i)$ into $A(C'_i)$ to obtain that
$$
|F'_i(z)| \leq \frac{4M}{m} \frac{r'_i}{r_i}
$$
for every $z \in C_i$. Exchanging the roles of $F_i$ and $F_i^{-1}$, we also have
$$
|(F_i^{-1})'(z)| \leq \frac{4M}{m} \frac{r_i}{r'_i}.
$$
for every $z \in C'_i$. 
Since $\iota_i=F_i|_{C_i}$, it follows from the above two inequalities that the derivative $|\iota'_i|$ is 
comparable to $r'_i/r_i$ uniformly, that is, there is some constant $K_1 \geq 1$ such that
$$
\frac{1}{K_1} \frac{r'_i}{r_i} \leq |\iota'_i(z)| \leq K_1 \frac{r'_i}{r_i} \qquad (\forall z \in C_i)
$$
for all $i \in \mathbb N$.

Next, we consider the derivative $|g'_i|$ for the M\"obius transformation $g_i$ on $C_i$.
Since the axis of $g_i$ is orthogonal to the hyperbolic planes in $\mathbb H^3$ spanned by $C_i$ and $C'_i$,
if $C_i$ and $C'_i$ are apart from each other uniformly, then the end points of the axis come closer to
the centers of $C_i$ and $C'_i$. In this case, $|g'_i|$ is 
comparable to the ratio $r'_i/r_i$ of their radii uniformly. By Lemma \ref{grotzsch}, any 
$C_i$ and $C'_i$ have mutually disjoint outer annuli whose conformal moduli are bounded away from 0
independent of $i \in \mathbb N$. This fact serves as a criterion for $C_i$ and $C'_i$ to be uniformly apart,
and hence we can find some constant $K_2 \geq 1$ such that
$$
\frac{1}{K_2} \frac{r'_i}{r_i} \leq |g'_i(z)| \leq K_2 \frac{r'_i}{r_i} \qquad (\forall z \in C_i)
$$
for all $i \in \mathbb N$.

From the above two inequalities, the derivative of the diffeomorphism $g_i^{-1} \circ \iota_i:C_i \to C_i$ 
is bounded from above by $K_1 K_2$ and from below by $1/(K_1 K_2)$ independent of $i \in \mathbb N$.
This completes the proof of the claim.
\qed
\medskip

\begin{rem}
{\rm A hyperbolic Riemann surface $R$ or its Fuchsian model $\Gamma$ on the unit disc 
$\Delta$ is called {\it conservative} if
the horocyclic limit set of $\Gamma$ has full measure on $\partial \Delta$. It is known that
if $R$ admits a uniform pants decomposition, or more generally, if the injectivity  radius  on $R$ is  uniformly bounded,
then it is conservative (see \cite[Th.5.11]{M-T}). In contrast, 
an example constructed in the next section is not conservative.}
\end{rem}

Regarding the uniqueness of the classical Schottky uniformization,
we have the following

\begin{thm}\label{uniqueness}
Suppose $G$  is a classical Schottky group with respect to the admissible configuration of circles $\mathcal{C}$ where   
 $\Omega(G)/G$ has  a bounded pants decomposition. Let 
 $G^{\prime}$ be a Schottky-like group with admissible configuration of 
 curves $\mathcal{C}^{\prime}$. If 
 $h:\Omega(G)/G \to \Omega(G')/G'$ is a conformal mapping which sends the  free homotopy classes of simple closed curves determined by 
$\mathcal C$  to those of $\mathcal{C}^{\prime}$,  then $G$ and $G'$ are conjugate by a M\"obius transformation.
\end{thm}

\begin{proof}
In the proof of Theorem \ref{thm: qc uniformization}, we used a result by
Maitani and Taniguchi \cite{M-T} claiming that $\Omega(G)$ 
belongs to the class $O_{AD}$, that is, holomorphic functions of finite
Dirichlet energy on $\Omega(G)$ are only constant functions. It is known that
this condition is equivalent to saying that any conformal map of $\Omega(G)$ 
into $\widehat{\mathbb C}$ 
extends to a M\"obius transformation. Since the conformal map $h$ lifts to $\Omega(G)$
and then extends to a M\"obius transformation, we see that this gives the conjugation between
$G$ and $G'$.
\end{proof}

The key fact in the above proof is the conformal removability of the limit set of
the classical Schottky group $G$,
and in a slightly different setting for the reflection group with respect to the circles
this is also a crucial property for
the circle domain theorem due to He and Schramm. They also conjectured conformal rigidity of
certain circle domains. Solutions of this conjecture for several important cases are given by
Yousi \cite{Y} and the joint paper \cite{N-Y} with Ntalampekos.

Let $M$ be an infinite genus topological handlebody. 
We consider the  problem of  whether the hyperbolic structure on $M$ is determined by the conformal structure on $\partial M$. The above observation yields 
the following assertion in the  special case of the conformal structure on $\partial M$ having a bounded pants decomposition. 


\begin{cor}\label{cor: uniqueness of handlebody}
Let $M$ be an infinite genus handlebody with  complete
meridian system $\mathcal A$ uniformized topologically by
a classical Schottky group $G$ with  admissible configuration $\mathcal C$ corresponding to $\mathcal A$. If $\Omega(G)/G$ admits a bounded pants decomposition,
then any conformal structure on $\partial M$ quasiconformally equivalent to $\Omega(G)/G$
relative to $\mathcal C$ determines the unique hyperbolic structure on $M$.
\end{cor}

\begin{proof}
We consider a quasiconformal deformation $G_1$ of $G$ by a quasiconformal homeo\-morphism of $\widehat {\mathbb C}$
compatible with $G$. Because the class $O_{AD}$ is preserved for planar domains under quasiconformal mappings
(see \cite[II.15.B]{SN}), $\Omega(G_1)$ also belongs to $O_{AD}$. Then, the statement follows from Theorem \ref{uniqueness}.
\end{proof}


\section{No Schottky uniformization: examples}
\label{sec: no Schottky uniformization: examples}

In this section we give necessary conditions for a Riemann surface of infinite genus with only  
non-planar ends to have a classical Schottky uniformization. We show also that for every Riemann surface of infinite genus,
there exists a topologically equivalent Riemann surface that does not have a Schottky uniformization.
 
\subsection{No quasiconformal classical Schottky uniformization}

Here is an example of a Riemann surface of infinite genus with only non-planar ends which is not 
quasiconformally equivalent to a classical Schottky uniformized surface.  

One simple observation is that if we choose the admissible circle configuration
for a classical Schottky group $G$ of infinite genus appropriately,
then the accumulation set $\Lambda'$ of the circles is homeomorphic to the space of ends of the Riemann surface (Lemma \ref{end}). 
Moreover, $\Lambda'$ is contained in the limit set of $G$ (Lemma \ref{lem: basics2}).
Since the limit set of $G$ is totally disconnected we have that $\Lambda'$ is totally disconnected. 

A rough construction of a Riemann surface $R$ of no such quasiconformal uniformization is as follows.
Assume that we have infinitely many disjoint circles $c_i$ in the unit disc $\Delta$ 
that accumulate on all of the unit circle $\partial \Delta$. 
Denote the common exterior of the circles $\{c_i\}$ in $\Delta$ by $\Delta'$.
We prepare a torus $T_i$ with one  boundary curve for each $i$
and glue $T_i$ to $\Delta'$ along $c_i$ identified with $\partial T_i$.
The resulting Riemann surface is $R$ of infinite genus with only one end.  
Such a Riemann surface is sometimes called a Loch Ness monster (see Figure \ref{Loch Ness}).
A survey on historical and recent developments of researches on these surfaces is in \cite{AR1}.


\begin{figure}[hbtp]
\centering
\AffixLabels{\centerline{\epsfig{file = 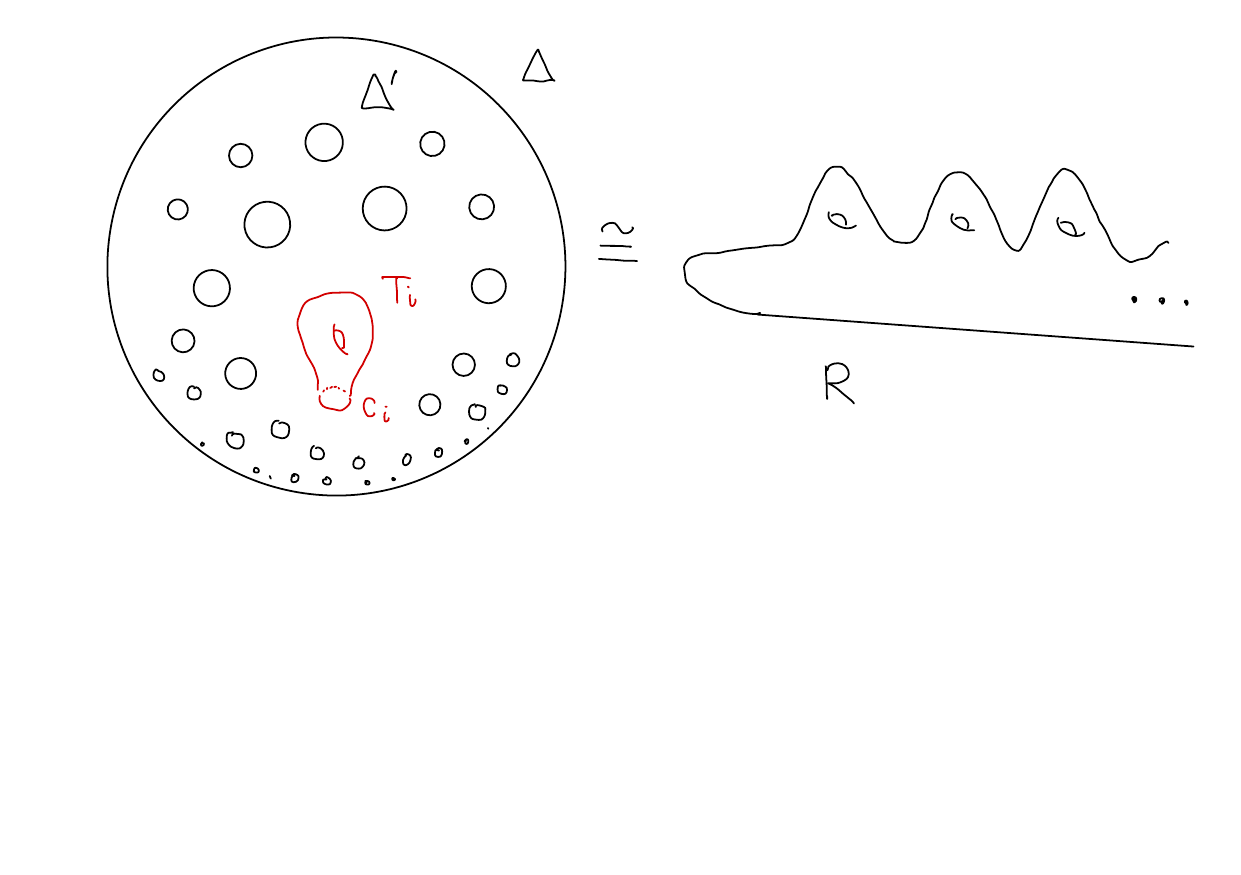,width=15.0cm} }}
\vspace{-3.5cm}
\caption{Loch Ness: Infinite genus with one end.}
\label{Loch Ness}
\end{figure}

Now suppose that $R$ is quasiconformally equivalent to a classical Schottky uniformization.  
We want to argue that this cannot happen. Let $G$ be the uniformizing Schottky group. 
By the construction of $R$, we see that the space of ends of $R$ is homeomorphic to
the set $\Lambda'$ of accumulation points of the admissible circle configuration $\mathcal C$ for $G$.
Since $R$ has only one end, $\Lambda'$ must be a single point, which is assumed to be  infinity $\infty$. 
We want to claim that there is no quasiconformal map between $R$ and $\Omega(G)/G$.
An intuition for this is based on the fact that the disc is not quasiconformally equivalent to the plane, 
but for a complete argument there remains a problem. 

We remark that by Theorem \ref{thm: classical schottky uniformize} there is a topological
uniformization of $R$ by a classical Schottky group $G$. An explicit
representation of generators of $G$ can be found in Arredondo and Ram\'irez Maluendas \cite{AR2}.

For the proof of the claim above, we use somewhat technical arguments based on the results in \cite{MS}.

\begin{definition}
{\rm
We say that a Riemann surface $R$ is of {\it S-type} if there is no other Riemann surface $\widehat R$
strictly larger than $R$ such that the inclusion map $\iota:R \hookrightarrow \widehat R$ is a homotopy equivalence.
We call such an extension of $R$ to $\widehat R$ homotopic extension.
}
\end{definition}

%

Concerning a Schottky uniformization, 
the following claim gives a sufficient condition for being of S-type.
We note that S-type is a quasiconformally invariant condition as is shown in \cite[Th.3]{MS}.

\begin{thm}\label{S-type}
Let $G$ be a  classical Schottky group 
with admissible configuration $\mathcal{C}$ of circles.
If the set $\Lambda'$ of accumulation of $\mathcal{C}$ is countable, then
the Riemann surface $R=\Omega(G)/G$ uniformized by $G$ is of S-type.
\end{thm}

\begin{proof}
Let $D=\Ext (\mathcal{C})^\circ \subset \widehat {\mathbb C}$ 
be the fundamental domain for $G$. The Riemann surface $R$ is obtained by
pasting the corresponding circles of $\mathcal{C}$, which are on $\partial D$. 

Suppose that $R$ is not of S-type. By a theorem of Sakai \cite{Sa} (see also the main theorem of \cite{MS}),
we have a simply connected domain $U \subset R$ 
(called a disc with D-ideal boundary)
and a conformal map $\varphi$ of $U$ into the unit disc $\Delta$ with
$\varphi(U) \subsetneq \Delta$ such that $\varphi$ extends continuously to the relative boundary $\partial_R U$ of
$U$ in $R$ with $\varphi(\partial_R U) \subset \partial \Delta$. Then,
adding $\Delta \setminus \varphi(U)$ to $R$ under the identification by $\varphi$, we obtain
a Riemann surface $\widehat R$ with
$R \subsetneq \widehat R$ such that
the inclusion map $\iota:R \hookrightarrow \widehat R$ is a homotopy equivalence.

From the above construction of $\widehat R$, we see that 
there is a bordered domain $\widehat D$ such that $D$ is conformally embedded strictly in $\widehat D$
keeping the correspondence of the boundary circles. Moreover, this embedding is a homotopy equivalence.
In these circumstances,
if we fill discs inside of all boundary circles of $D$, the resulting Riemann surface, which is
$\widehat {\mathbb C}$ minus a countable many points, can be strictly embedded into a larger Riemann surface
in a homotopically equivalent way. However, we know that $\widehat {\mathbb C}$ minus a countable many points
admits no Green function as a Riemann surface. 
On the other hand, admitting a homotopic extension to a strictly larger Riemann surface, 
this is not of S-type, and in particular 
it admits the Green function by \cite[Th.4]{MS}. This is a contradiction, which proves that $R$ is of S-type.
\end{proof}

In order to verify the above assumption on $\Lambda'$ for the admissible configuration of circles $\mathcal C$,
we show the following claim by means of Lemma \ref{end}.

\begin{lem}\label{end2}
Let $R$ be a Riemann surface of infinite genus with only one end (non-planar).
Assume that there is a family of disjoint simple closed geodesics $\{\gamma_j\}_{j \in \mathbb N}$ in $R$
that satisfies the following conditions:
\begin{itemize}
\item[(1)]
$R \setminus \bigcup_j \gamma_j$ is a planar domain $D$, which we regard as a bordered hyperbolic surface;
\item[(2)]
the hyperbolic distances in $D$ between the boundary components $L_j$ and $L'_j$ of $\partial D$
corresponding to $\gamma_j$ are uniformly bounded;
\item[(3)]
the hyperbolic lengths of $\gamma_j$ are uniformly bounded.
\end{itemize}
Let $G$ be any classical Schottky group uniformizing $R$.
Then, the set $\Lambda'$ of accumulation 
of the admissible  circle configuration $\mathcal{C}=\{C_i ,C_i^{\prime}\}_{i \in \mathbb N}$ 
for $G$ is homeomorphic to the end space of 
$R=\Omega(G)/G$.
\end{lem}

\begin{proof}
Let $G$ be a classical Schottky group with admissible circle configuration $\mathcal{C}=\{C_i ,C_i^{\prime}\}$
uniformizing $R$. We will show that the hyperbolic distances $d(C_i, C'_i)$ between $C_i$ and $C_i^{\prime}$ in $\Omega(G)$
are uniformly bounded. Then, it is easy to see that assumption \eqref{*} of Lemma \ref{end} for $\mathcal{C}$ is satisfied,
and the statement follows from this lemma.

We consider the simple closed curves $c_i$ on $R$ corresponding to $C_i ,C_i^{\prime}$.
If $c_i$ is freely homotopic to some $\gamma_j$, then by assumption (2),
$d(C_i, C'_i)$ is bounded by some uniform constant.
If $c_i$ is freely homotopic to no such $\gamma_j$, then it should intersect some $\gamma_j$ transversely.
In this case, $d(C_i, C'_i)$ is uniformly bounded by assumption (3).
\end{proof}

We remark that the existence of a curve family $\{\gamma_j\}$ in $R$ satisfying the above assumptions
is a quasiconformally invariant condition.

\begin{exam}\label{Ex1}{\rm
We will construct a Riemann surface $R$ of infinite genus with only one end (non-planar) that is not of S-type
and that satisfies the assumptions of Lemma \ref{end2}.
Then by Theorem \ref{S-type}, $R$ is quasiconformally equivalent to
no Riemann surface uniformized by a classical Schottky group.

Let $\Delta$ be the unit disc and let $U \subsetneq \Delta$ be a Jordan domain
such that the relative boundary $\partial_\Delta U \subset \Delta$ consists of
countably many Jordan curves and $E=\partial U \cap \partial \Delta$ is totally disconnected.
Moreover, we can choose $U$ so that
it is a disc with D-ideal boundary, that is,
if we map $U$ conformally onto the unit disc, 
the image of $E$ on the unit circle does not belong to class $N_D$.
Here a compact subset on $\mathbb C$ belongs to $N_D$ 
if its complement in $\mathbb C$ admits no non-constant Dirichlet finite holomorphic functions.

We choose infinitely many disjoint circles $\{c_i\}$ in $\Delta \setminus \overline U$ 
that accumulate on all of the unit circle $\partial \Delta$. Let $D$ be the common exterior of
these circles in $\Delta$.
Then, we choose once-holed tori $\{T_i\}$ of uniform geometry and paste each $T_i$ along $c_i$
as explained before so that the resulting Riemann surface $R$ satisfies the assumptions of Lemma \ref{end2}.
Since $R$ contains a disc $U$ with D-ideal boundary, it is not of S-type.

Suppose that there is a classical Schottky group $G$ such that $R^*=\Omega(G)/G$ is quasiconformally
equivalent to $R$. We see that $R^*$ is not of S-type and $R^*$ satisfies the assumptions of 
Lemma \ref{end2}. However, Theorem \ref{S-type} together with this lemma shows that 
$R^*$ is of S-type. This contradiction implies that there is no such classical Schottky group $G$.}
\end{exam}

\subsection{Hyperbolic structures on a topological surface that are not unifor\-mized by a Schottky group}
We next give an alternative construction, for each topologically infinite type,  of a Riemann surface that can not have a Schottky uniformization.  We start with a necessary condition for a Riemann surface to have a Schottky uniformization. 
Suppose $X=\mathbb{H}/\Gamma$  is a Riemann surface, where
$\Gamma$ is a torsion-free Fuchsian group. The convex core of 
$X$, denoted by $C(X)$,  is the smallest geodesic closed subsurface containing all the closed geodesic.  While $C(X)$ may not be geodesically complete it does have a geometric completion by attaching possibly funnels or half-planes (see \cite{B-S}).   As a result,
since funnels also have embedded half-planes we have that if $C(X) \subsetneq X$, then $X$ contains a conformally 
embedded half-plane.  Note that $C(X)=X$ if and only if $\Gamma$ is a Fuchsian group of the first kind (that is,
the limit set of $\Gamma$ is all of the circle).

\begin{lem}\label{lem: geometric criteria}
If $S$ is an infinite genus Riemann surface with only non-planar ends, then any quasiconformal Schottky uniformization of $S$
satisfies $C(S)=S$ (that is, the Fuchsian group for $S$ is of the first kind).
\end{lem}

\begin{proof}
Suppose $X=\Omega(G) /G$ is a Schottky uniformization.
Now if $C(X) \subsetneq X$, then $X$ contains a conformally embedded hyperbolic half-plane $P$ (\cite[Cor.3.6]{B-S}). Since $P$ is simply connected it lifts to the Schottky cover $\Omega(G)$ as a conformally embedded hyperbolic half-plane. But this is a contradiction since $\Omega(G)$ is, by the definition of Schottky groups, 
$\widehat{\mathbb{C}}$ with a Cantor set removed. Thus, we have $C(X)=X$. If $S$ is quasiconformally equivalent to $X$, then
$C(S)=S$ by the fact that being of the first kind is a quasiconformal invariant.
\end{proof}

\begin{exam}\label{Ex2}
{\rm 
First, for each topologically infinite type surface construct a 
Riemann surface $\mathbb{H}/\Gamma$ in its homeomorphism class, where $\Gamma$ is of the  second kind. Roughly speaking, to construct such a surface start with a topological pants decomposition of the surface and choose a flute subsurface on it. Every infinite type surface has one. Now make the pants curves    geometric by setting the pants cuff lengths of the flute subsurface to get very large with zero twisting along the cuffs. The geometry on the complement  of the flute subsurface is arbitrary.  In this  case, by construction  there must be a nested sequence of geodesics that converge to a simple geodesic  which bounds a half-plane.  This guarantees that $\Gamma$ is of the second kind (see \cite{Ba} or \cite{B-S} for details).
Thus, this  Riemann surface can not be quasiconformally Schottky uniformized with respect to any admissible circle set by Lemma \ref{lem: geometric criteria}.
}
\end{exam}

\end{document}